\documentclass[11 pt]{amsart}
\usepackage[T1]{fontenc}
\usepackage[utf8]{inputenc}
\usepackage[english]{babel}
\usepackage{mathtools}
\usepackage{amsfonts}
\usepackage{amsthm}
\usepackage{braket}
\usepackage{lipsum}
\def\Mat{\mathbb{M}}
\def\Ball{\mathcal{B}}
\def\A{\mathcal{A}}
\def\h{\mathrm{H}}
\def\D{\mathbb{D}}
\def\N{\mathbb{N}}
\def\M{\mathcal{M}}
\def\C{\mathbb{C}}

\def\p{\mathbb{P}}
\def\bound{\mathcal{B}}

\def\diag{\mathrm{diag}}
\def\de{\partial}

\def\spa{\mathrm{span}}

\def\T{\mathbb{T}}

\def\rkhs{\mathcal{H}}
\def\W{\mathbb{W}}
\numberwithin{equation}{section}
\theoremstyle{plain}
\newtheorem{theo}{Theorem}[section]

\newtheorem{prop}[theo]{Proposition}
\newtheorem{lemma}[theo]{Lemma}
\newtheorem{defi}[theo]{Definition}
\theoremstyle{definition}
\newtheorem{question}{Question}
\newtheorem{ex}[theo]{Example}
\newtheorem{rem}[theo]{Remark}
\title{Interpolating $d$-tuples of Matrices}
\author{Alberto Dayan}
\address{Department of Mathematics\newline Washington University in St. Louis,\newline One Brookings Drive, St. Louis, MO 63130, USA}
\email{alberto.dayan@wustl.edu}
\thanks{The author was partially supported by National Science Foundation Grant DMS 1565243}
\begin{document}
\maketitle
\begin{abstract}
We study interpolating sequences of $d$-tuples of matrices, by looking at the commuting and the non-commuting case separately. In both cases, we will give a characterization of such sequences in terms of separation conditions on suitable reproducing kernel Hilbert spaces, and we will give sufficient conditions stated in terms of separation via analytic functions. Examples of such interpolating sequences will also be given.
\end{abstract}
\section{Introduction}
Let $\h^\infty$ be the algebra of bounded analytic functions on the unit disc, normed with 
\[
||\phi||_\infty:=\sup_{z\in\D}|\phi(z)|\qquad\phi\in\h^\infty.
\]
A sequence $Z:=(z_n)_{n\in\N}$ in the unit disc $\D$ is interpolating if, for any bounded sequence $(w_n)_{n\in\N}$ in $\C$, there exists a bounded analytic function $\phi$ such that $\phi(z_n)=w_n$, for any $n$ in $\N$. Intuitively, an interpolating sequence is a separated sequence, as we want to specify the values of a bounded holomorphic functions \emph{arbitrarily} at the nodes. Since we work with analytic functions in the unit disc, a natural way to measure the distance between the points of a sequence is by using the {\bf pseudo-hyperbolic distance}
\[
\rho(z, w)=|b_w(z)|:=\left|\frac{w-z}{1-\overline{w}z}\right|\qquad z,w\in\D.
\]
Here $b_w$ denotes the involutive Blaschke factor at $w$.\\
The sequence $Z$ is said to be {\bf strongly separated} if 
\[
\inf_{n\in\N}\prod_{j\ne n}\rho(z_n, z_j)>0
\]
whereas $z$ is {\bf weakly separated} if 
\[
\inf_{n\ne j}\rho(z_n, z_j)>0.
\]
Since any function $\phi$ in $\h^\infty$ with zeros of multiplicities $m_1, \dots, m_l$ at the points $\lambda_1, \dots, \lambda_l$ in $\D$ is divisible in $\h^\infty$ by a Blaschke product
\[
\phi=g~\prod_{j=1}^ lb_{\lambda_j}^ {m_j},
\]
where $||g||_\infty=||\phi||_{\infty}$, being weakly separated is equivalent to the existence, for any pair of distinct indices $n$ and $j$, of a function $\phi_{n, j}$ in the unit ball of $\h^\infty$ that vanishes at $z_j$ and such that $|\phi_{n, j}(z_n)|\geq\delta$, for some positive $\delta$ independent of $n$ and $j$. Analogously, $Z$ is strongly separated if and only if there exists a positive $\delta$ such that, for any $n$ in $\N$, there exists a function $\phi_n$ in the unit ball of $\h^\infty$ that vanishes at all points of the sequence except $z_n$ and $|\phi_n(z_n)|\geq\delta$.\\

Let $\h^2$ be the Hardy space, that is, the Hilbert space of analytic functions on the unit disc with square summable Taylor coefficients. Carleson characterized interpolating sequences in the unit disc in \cite{carlocm} and \cite{carloss}:
\begin{theo}
\label{theo:carleson}
Let $Z$ be a sequence in the unit disc. The following are equivalent:
\begin{description}
\item[(i)] $Z$ is interpolating;
\item[(ii)] $Z$ is strongly separated;
\item[(iii)] $Z$ is weakly separated and the measure
\[
\mu_Z:=\sum_{n\in\N}(1-|z_n|^2)\delta_{z_n}
\]
satisfies the embedding condition
\begin{equation}
\label{eqn:cmh2}
||f||_{\mathrm{L}^2(\D, \mu_Z)}\leq C~||f||_{\h^2}\qquad f\in\h^2.
\end{equation}
\end{description}
\end{theo}
Condition \eqref{eqn:cmh2} can be extended to more general settings, as it can be expressed in terms of the Euclidean geometry of {\bf reproducing kernel Hilbert spaces}. A reproducing kernel Hilbert space on a domain $X$ of $\C^d$ is a Hilbert space $\rkhs$ of functions on $X$ such that point evaluation at any $x$ in $X$ is continuous. Thanks to Riesz representation Theorem, this implies that, for any $x$ in $X$, there exists a function $k_x$ in $\rkhs$ such that 
\begin{equation}
\label{eqn:repr}
f(x)=\braket{f, k_x}_\rkhs\qquad f\in\rkhs.
\end{equation}
This define a function $k\colon X\times X\to\C$ by
\[
k(x, y)=k_y(x)
\]
which is referred as a \emph{kernel} on $X$. To emphasize the dependence on the kernel $k$ of the Hilbert space $\rkhs$, we will denote the associated reproducing kernel Hilbert space by $\rkhs_k$, where necessary. Each reproducing kernel Hilbert space $\rkhs_k$ has a so called {\bf multiplier algebra}
\[
\M_k:=\{\phi\colon X\to\C \,|\, \phi f\in\rkhs_k,\,f\in\rkhs_k\}
\]
defined as the algebra of those functions on $X$ that multiplies $\rkhs_k$ into itself. Thanks to the closed graph Theorem, any function $\phi$ in $\M_k$ induces a bounded linear operator $M_\phi$ on $\rkhs_k$ by
\[
M_\phi(f):=\phi~f\qquad f\in\rkhs_k.
\]
In particular, $\M_k$ is normed by
\[
||\phi||_{\M_k}:=||M_\phi||_{\bound(\rkhs_k)}.
\]
 The Hardy space $\h^2$ is one of the most studied examples of a reproducing kernel Hilbert space, together with its {Szegö kernel}
\[
s_w(z):=\frac{1}{1-\overline{w}z}\qquad z, w\in\D.
\]
Its multiplier algebra $\M_s$ coincides isometrically with $\h^\infty$. \\
A measure $\mu$ on $X$ is a {\bf Carleson measure} for $\rkhs$ if $\mathrm{L}^2(X, \mu)$ embeds continuously in $\rkhs$. Hence \eqref{eqn:cmh2} says precisely that $\mu_Z$ is a Carleson measure for $\h^2$.\\
It turns out that separation conditions on a sequence $Z$ correspond to separation conditions on the kernel functions at the point of such sequence. More precisely, let 
\[
\hat{h}:=\frac{h}{||h||}
\]
be the normalization of any element in a Hilbert space $\rkhs$, and recall that a sequence $(h_n)_{n\in\N}$ of unit vectors in $\rkhs$ is a {\bf Riesz system} if there exists a positive $C$ such that, for any $(a_n)_{n\in\N}$ in $l^2$, 
\begin{equation}
\label{eqn:rs}
\frac{1}{C^2}~\sum_{n\in\N}|a_n|^2\leq\left|\left|\sum_{n\in\N}a_n h_n\right|\right|_\rkhs^2\leq C^2~\sum_{n\in\N}|a_n|^2.
\end{equation}
If the right inequality holds, we say that $(h_n)_{n\in\N}$ is a {\bf Bessel system}. The least constant $C$ in \eqref{eqn:rs} is denoted by the \emph{Riesz constant} of the sequence $(h_n)_{n\in\N}$. Shapiro and Shields proved in \cite{shapiro} that a sequence $Z$ in $\D$ is interpolating if and only if the sequence $(\hat{s}_{z_n})_{n\in\N}$ of normalized Szegö kernels at the points of the sequence is a Riesz system. Moreover, one can find in \cite[Th. 9.5]{john} a proof that, given any reproducing kernel Hilbert space $\rkhs_k$ on a domain $X$ and given a sequence $(x_n)_{n\in\N}$ of points in $X$, then the measure
\[
\mu_X:=\sum_{n\in\N}||k_{x_n}||^{-2}_{\rkhs_k}\delta_{x_n}
\]
is a Carleson measure for $\rkhs$ if and only if the sequence $(\hat{k}_{x_n})_{n\in\N}$ is a Bessel system in $\rkhs$. In particular, \eqref{eqn:cmh2} can be restated as the sequence of normalized Szegö kernels at the points of $Z$ being a Bessel system in the Hardy space. Such results, other than constitute a valuable correspondence between separation of points via holomorphic functions on a domain $X$ and separation of the respective kernel functions in a reproducing kernel Hilbert space, are extremely useful for extending Carleson's Theorem to more general settings.\\

One of the various efforts meant to extend Theorem \ref{theo:carleson} can be found in \cite{dayoss} and \cite{dayocm}, where the author extended Carleson interpolation Theorem to sequences of square matrices (of any dimensions). A holomorphic function $f$ on the unit disc is evaluated at a square matrix $M$ via the following extension of Cauchy integral formula:
\begin{equation}
\label{eqn:cauchymatonevar}
f(M):=\frac{1}{2\pi i}\int_{\de\D}f(\xi)(\xi~Id-M)^{-1}\,d\xi.
\end{equation}
In particular, for $f(M)$ to be defined, the eigenvalues of $M$ must lay in $\D$. Given then any sequence $A:=(A_n)_{n\in\N}$ of square matrices (of possibly different and not uniformly bounded dimensions) with spectra in the unit disc, one can extend the classic definition of interpolating sequences by identifying a sequence of bounded targets with a bounded sequence in $\h^\infty$:
\begin{defi}
\label{defi:intmat}
$A$ is interpolating if, for any bounded sequence $(\phi_n)_{n\in\N}$ in $\h^\infty$ there exists a bounded analytic function $\phi$ such that 
\[
\phi(A_n)=\phi_n(A_n)\qquad n\in\N.
\]
\end{defi}
As for separating the matrices in $A$, one can follow the same outline as for scalars, and define $A$ to be {\bf weakly separated} if there exists a $\delta>0$ such that, for any pair of distinct indices $j$ and $n$, there exits a function $\phi_{n, j}$ in the unit ball of $\h^\infty$ such that 
\[
\phi_{n, j}(A_j)=0\quad\phi(A_n)=\delta~Id.
\]
In the same fashion, $A$ is said to be {\bf strongly separated} if there exists a bounded sequence of $\h^\infty$ functions $(\phi_n)_{n\in\N}$ such that
\[
\phi_n(A_j)=\delta_{n, j}~Id. 
\]
In order to extend the Carleson measure condition in \eqref{eqn:cmh2} to the sequence of square matrices $A$, one has to consider separation conditions on finite dimensional subspaces of $\h^2$. For any $n$ in $\N$, let
\begin{equation}
\label{eqn:model}
H_n:=\h^2\ominus\{f\in\h^2\,|\,f(A_n)=0\}
\end{equation}
be the orthogonal complement in $\h^2$ of the subspace of all those $\h^2$ functions that vanish at $A_n$. Observe that if $A_n=z_n$ is a $1\times1$ scalar in $\D$, then $H_n$ is the line spanned by the Szegö kernel at $z_n$, thanks to the reproducing property in \eqref{eqn:repr}. With this in mind, one can extend the notion of Riesz and Bessel system to any sequence $(\rkhs_n)_{n\in\N}$ of closed subspaces of a Hilbert space $\rkhs$, by asking that there exists a positive $C$ so that, for any sequence of unit vectors $(h_n)_{n\in\N}$ chosen in $\rkhs$ so that each $h_n$ belongs to $\rkhs_n$, \eqref{eqn:rs} holds (or, in the case of a Bessel system, the right inequality in \eqref{eqn:rs} holds). Carleson's Theorem and Shapiro and Shields result in \cite{shapiro} can be then extended to this matrix node interpolation problem as follows, \cite{dayoss} \cite{dayocm}:
\begin{theo}
Let $A$ be a sequence of square matrices with spectra in the unit disc. The following are equivalent:
\begin{description}
\item[(i)] $A$ is interpolating;
\item[(ii)] $A$ is strongly separated;
\item[(iii)] $A$ is weakly separated and the sequence $(H_n)_{n\in\N}$ defined in \eqref{eqn:model} is a Bessel system in $\h^2$;
\item[(iv)] $(H_n)_{n\in\N}$ is a Riesz system in $\h^2$.
\end{description}
\end{theo}
The aim of this note is to consider sequences of \emph{$d$-tuples} of matrices (of any dimensions), and to determine whether some well known results on interpolating sequences in \emph{multi-variable} function theory extend to a matrix interpolation problem. The main distinction for such a multi-variable matrix interpolation problem from the scalar case is that a $d$-tuple
\[
M=(M^1, \dots, M^d)
\]
of $m\times m$ matrices might commute or not. If the matrices in $M$ commute, then \eqref{eqn:cauchymatonevar} has a multi-variable analogue, and a bounded holomorphic function on the polydisc can be applied to $M$, provided that its joint spectrum belongs to $\D^d$. On the other hand, not even a polynomial can be applied to $M$ if its components do not commute, and therefore in order to study interpolating properties of sequences of $d$-tuples of non commuting matrices one has to consider a whole different class of functions. Given this intrinsic difference, this introduction will treat those two cases separately.
\subsection{Interpolating $d$-tuples of Commuting Matrices}
 Let $\h^\infty_d$ be the Banach algebra of bounded analytic functions on the polydisc $\D^d$. A sequence $Z=(z_n)_{n\in\N}$ in $\D^d$ is interpolating for $\h^\infty_d$ if, given any bounded $(w_n)_{n\in\N}$ in $\C$ there exists a bounded holomorphic function $f$ on $\D^d$ so that $f(z_n)=w_n$, for all $n$. Berndtsson, Chang and Lin proved in \cite{bern} the following analogue of Theorem \ref{theo:carleson}:
\begin{theo}[Berndtsson, Chang and Lin]
\label{theo:poly:interp}
Let $Z=(z_n)_{n\in\N}$ be a sequence in $\D^d$, and let \textnormal{(a)}, \textnormal{(b)} and \textnormal{(c)} denote the following statements:
\begin{itemize}
\item[(a)]
\begin{equation}
\label{eqn:ss}
\inf_{n\in\N}\prod_{j\ne n}\rho_G(z_n, z_j)>0;
\end{equation}
\item[(b)] $Z$ is interpolating for $\h^\infty_d$;
\item[(c)] The measure
\[
\mu_Z:=\sum_{n\in\N}\left(\prod_{i=1}^d(1-|z_n^i|^2)\right)\delta_{z_n}
\]
is a Carleson measure for $\h^2_d$ and
\begin{equation}
\label{eqn:ws}
\inf_{n\ne j}\rho_G(z_n, z_j )>0.
\end{equation}
\end{itemize}
Then \textnormal{(a)}$\implies$\textnormal{(b)}$\implies$\textnormal{(c)}, and if $d\geq2$ none of the converse implications hold.
\end{theo}
$\h^2_d$ will be here the \emph{Hardy space} on the polydisc, that is, the reproducing kernel Hilbert space on $\D^d$ with kernel
\[
s_d(z, w):=\prod_{i=1}^d\frac{1}{1-\overline{w^i}z^i}\qquad z, w\in\D^d.
\]
Conditions \eqref{eqn:ss} and \eqref{eqn:ws} are separation conditions, both stated in terms of the so called \emph{Gleason distance} on the polydisc:
\[
\rho_G(z, w):=\max_{i=1,\dots d}\rho(z^i, w^i)\qquad z, w\in\D^d.
\]
It turns out that $\rho_G$ corresponds to separating points via bounded analytic function, since
\begin{equation}
\label{eqn:gleason}
\rho_G(z, w)=\sup\{|\phi(z)|\,|\, ||\phi||_\infty\leq1, \phi(w)=0\}.
\end{equation}
Throughout this note, \eqref{eqn:ss} will refer to {\bf strong separation} on the polydisc, while \eqref{eqn:ws} defines a {\bf weakly separated} sequence on the polydisc.\\
Even if Theorem \ref{theo:poly:interp} does not characterize interpolating sequences in the polydisc, the case $d=2$ provides a characterization of interpolating sequences by looking at the Euclidean geometry of those reproducing kernels Hilbert spaces whose multiplier algebra is $\h^\infty_2$. The main difference with the one-variable case is the fact that, rather than working with just the Szegö kernel, in order to obtain interpolating properties one has to consider a whole class of different kernels on the polydisc. In \cite{agler}, Agler and McCarthy characterized interpolating sequences in the bi-disc in terms of separation conditions on the class of so called {\bf admissible kernels}. A kernel $k$ on $\D^d$ is admissibile if the multiplications by the coordinates
\[
M_{z^1}, \dots, M_{z^d}
\]
form a set of commuting contractions on $\rkhs_k$. Let $\A_d$ be the set of all admissible kernels on $\D^d$, and let $\bound_d$ be the set of all kernels $k$ on the $d$-dimensional polydisc whose multiplier algebra coincide with $\h^\infty_d$. Since the coordinate functions are clearly in the unit ball of $\h^\infty$, we have that $\bound_d\subseteq\A_d$. Conversely, thanks to Ando's inequality \cite{ando}, for any kernel in $\bound_2$
\[
||\phi(T_1, T_2)||_{\bound(\rkhs_k)}\leq||\phi||_\infty
\] 
for any $\phi$ in $\h^\infty_2$ and for any pair $(T_1, T_2)$ of commuting contractions on $\rkhs_k$. Since
\[
M_\phi=\phi(M_{z^1}, M_{z^2})\qquad\phi\in\h^\infty_2
\]
we have that 
\[
\A_2=\bound_2.
\]
Namely, the class of admissible kernels coincides with the class of kernels on $\D^2$ whose multiplier algebra is $\h^\infty_2$.\\
\begin{ex}
Let $\alpha:=(\alpha_1, \dots, \alpha_d)$ be a $d$-tuple of positive integers. Then the kernel
\[
s^{\alpha}_w(z):=\prod_{i=1}^d\frac{1}{(1-\overline{w^i}z^i)^{\alpha_i}}\qquad z, w\in\D^d
\]
is admissible on $\D^d$. Indeed, multiplication by the coordinates being a set of contractions is equivalent to assert that, for any $z_1, \dots, z_N$ in $\D^d$,  the matrices
\[
[(1-\overline{z_{n}^i}z_{j}^i)k_{z_j}(z_n)]_{n, j=1}^N\qquad i=1, \dots, d
\]
are positive semi-definite.  Therefore, since the Szegö kernel is admissible and the Schur (point-wise) product of positive semi-definite matrices is positive semi-definite, $s^{\alpha}$ belongs to $\A_d$ for any $\alpha$.
\end{ex}
Let, for any kernel $k$ in $\A_2$
\[
S^{k}_Z:=\underset{n\in\N}{\overline{\spa}}\{k_{z_n}\}\subseteq\rkhs_k.
\]
Given a bounded sequence $(w_n)_{n\in\N}$, if there exits a function $\phi$ in the unit ball $\h^\infty_2$ such that 
\[
\phi(z_n)=w_n\qquad n\in\N,
\]
then, for any $k$ in $\A_2$, the operator
\[
T^k\colon S^{k}_Z\to S^{k}_Z
\]
such that $$T^k(k_{z_n})=\overline{w_n}k_{z_n}$$ is a contraction in $\rkhs_k$. Indeed, $T^k$ is the restriction of $M^*_\phi$ to $S^k_Z$ since, for any reproducing kernel Hilbert space $\rkhs_k$  on  domain $X$, the reproducing property of the kernel $k$ implies that every kernel function is an eigenvector of the adjoint of multiplication by any multiplier:
\[
\M_\phi^*(k_x)=\overline{\phi(x)}~k_x\qquad x\in X,\,\phi\in\M_k.
\]
Conversely, \cite{aglerpp}, if 
\[
\sup_{k\in\A_2}||T^k||_{\bound(\rkhs_k)}\leq1,
\] 
then \emph{each} $T^k$ is in fact the restriction of the adjoint of the multiplication by a function $\phi$ in the unit ball of $\h^\infty_2$, and $\phi(z_n)=w_n$. This, together with Theroem \ref{theo:nikriesz} below, gives the following characterization for interpolating sequences for the bidisk, which is part of a result in \cite{agler}:
\begin{theo}
\label{theo:bidisk}
Let $Z=(z_n)_{n\in\N}$ be a sequence in $\D^2$. Then $Z$ is interpolating for $\h^\infty_2$ if and only if there exists a constant $C\geq1$ such that, for any $k$ in $\A_2$, the sequence of normalized kernels $(\hat{k}_{z_n})_{n\in\N}$ is a Reisz sequence in $\rkhs_k$ with Riesz bound $C$.
\end{theo}
As we will see, the  key for the identification between Riesz systems of kernels functions and interpolating sequences for those reproducing kernel Hilbert spaces with extension properties is the following characterization of the Riesz system condition \cite[Th. 3.1.4 ]{nik}:
\begin{theo}
\label{theo:nikriesz}
Let $H=(H_n)_{n\in\N}$ be a sequence of closed sub-spaces of a Hilbert space $\rkhs$. The following are equivalent:
\begin{description}
\item[(i)] $H$ is a Riesz system with Riesz bound $C$;
\item[(ii)] For any sequence of linear functions $(\chi_n)_{n\in\N}$ such that 
\[
\chi_n\colon H_n\to H_n\qquad n\in\N
\]
and $\sup_{n\in\N}||\chi_n||\leq 1$, then 
\[
\chi\colon\underset{n\in\N}{\overline{\spa}}\{H_n\}\to\underset{n\in\N}{\overline{\spa}}\{H_n\}
\]
such that
\[
\chi_{|H_n}=\chi_n
\]
 is bounded by $C$.
\item[(iii)] For any sequence $(w_n)_{n\in\N}$ in the unit ball of $l^\infty$ the linear function
\[
\mu\colon\underset{n\in\N}{\overline{\spa}}\{H_n\}\to\underset{n\in\N}{\overline{\spa}}\{H_n\}
\]
such that $\mu_{|H_n}=w_n~Id_{H_n}$ is bounded by $C^2$.
\item[(iv)] For any finite subset $\sigma$ of $\N$, the linear function $P_\sigma\colon\underset{n\in\N}{\overline{\spa}}\{H_n\}\to\underset{n\in\N}{\overline{\spa}}\{H_n\}$  such that
\[
P_\sigma(x):=\begin{cases}
x\quad\text{if}&\quad x\in H_j, j\in\sigma\\
0\quad\text{if}&\quad x\in H_j, j\notin\sigma
\end{cases}
\]
is bounded by $\frac{1}{6}~ C^2$.
\end{description}
\end{theo}
Let $M=(M^1,\dots, M^d)$ be a $d$-tuple of matrices of size $m$, and let $\phi$ be a function in $\h^\infty_d$. In order for 
\[
\phi(M)=\phi(M^1, \dots, M^d)
\]
to makes sense the matrices in $M$ must be {\bf commuting}, so that the value of a bounded analytic function at $M$ can be expressed via the matrix version of Cauchy integral formula:
\begin{equation}
\label{eqn:cauchymat}
\phi(M):=\frac{1}{(2\pi i)^d}\int_{\T^d}f(\xi^1, \dots, \xi^d)(\xi^1~Id-M^1)^{-1}\dots(\xi^d~Id-M^d)^{-1}\, d\xi^1\dots d\xi^d,
\end{equation}
where $\T^d$ is the $d$-dimensional torus. As for the one dimensional case, \eqref{eqn:cauchymat} makes sense provided that the {\bf joint spectrum} of $M$ belongs to $\D^d$. In order to define the joint spectrum of $M$, recall that commuting matrices preserve each others invariant sub-spaces: consequently, $M$ can be \emph{jointly upper-triangularizable}, that is, there exits an $m\times m$ non singular matrix $P$ such that 
\begin{equation}
\label{eqn:jointr}
N^i=P^{-1}M^{i}P
\end{equation}
is upper triangular, for any $i=1, \dots, d$. The joint spectrum of $M$ is the (multi)-set
\[
\sigma(M)=\{(N^1_{j, j}, \dots, N^d_{j, j})\,|\, j=1, \dots, d\}\subset\C^d.
\]
In particular, $\lambda=(\lambda^1, \dots, \lambda^d)$ belongs to $\sigma(M)$ if and only if there exists a vector $x$ in $\C^m$ that is a joint eigenvector for the components of $M$:
\[
M^ix=\lambda^i x\qquad i=1, \dots, d.
\]

\begin{defi}
\label{defi:dintmat}
 A sequence $A=(A_n)_{n\in\N}$ of $d$-tuples of commuting matrices with joint spectra in $\D^d$ is {\bf interpolating} for $\h^\infty_d$ if for any bounded target sequence $(\phi_n)_{n\in\N}$ in $\h^\infty_d$ there exists a bounded analytic function $\phi$ on the polydisc such that
\[
\phi(A_n)=\phi_n(A_n)\qquad n\in\N.
\]
\end{defi}
In order to separate the matrices in $A$ we will extend \eqref{eqn:gleason} to this matrix setting by defining the Gleason distance $\rho_G$ between two $d$-tuples of commuting matrices $M$ and $N$ (of eventual different dimension) to be the supremum of all $|r|>0$ such that there exists a function $\phi$ in the unit ball of $\h^\infty$ so that 
\[
\phi(N)=0,\qquad \phi(M)=r~Id.
\]
 We will say then that the sequence $A$ is {\bf weakly separated} if 
\[
\inf_{n\ne j}\rho_G(A_n, A_j)>0,
\]
whereas \eqref{eqn:ss} extends to define {\bf strongly separated} sequences of commuting $d$-tuples by
\[
\inf_{n\in\N}\prod_{j\ne n}\rho_G(A_n, A_j)>0.
\]
If $\tau$ is any subset of $\N$ not containing $n$, then the distance between $A_n$ and $A_\tau:=(A_j)_{j\in\tau}$ is the inverse of the least norm of a function in $\h^\infty_d$ that vanishes on $(A_j)_{j\in\N}$ and it's the identity on $A_n$:
\begin{equation}
\label{eqn:dsepseq}
\rho_G(A_n, A_\tau):=\frac{1}{\inf\{||\phi||_\infty\,|\,\phi(A_n)=Id, \phi(A_j)=0, j\in\tau\}}.
\end{equation}
As for the one dimensional case, one can relate interpolation conditions for the sequence $A$ to separation condition on well chosen sub-spaces of some reproducing kernel Hilbert spaces, and the main difference with the theory in one variable is that one has to consider \emph{all} the kernels in $\bound_d$, rather than only the Szegö kernel. Let then $k$ be a kernel in $\bound_d$, and let $M$ be a $d$-tuple of commuting matrices with joint spectra in $\D^d$. Assume that a function $f$ in $\rkhs_k$ vanishes at $M$: without loss of generality, we can assume thanks to \eqref{eqn:jointr} that each $M^i$ is upper triangular. Let $\{z_1, \dots, z_d\}$ be the joint spectrum of $M$. For any $\xi$ in $\T^d$, 
\begin{equation}
\label{eqn:cauchy:factor}
(\xi^i~Id-M^i)^{-1}
\end{equation}
has on its main diagonal multiples of
\[
\frac{1}{(\xi^i- z^i_j )}\qquad j=1, \dots, d
\]
while on its second to main diagonal has a linear combination of factors of the form
\[
\frac{1}{(\xi^i- z^i_j )^2}\qquad j=1, \dots, d,
\]
where such linear combination depends exclusively on the algebraic structure of $M$, and not on the kernel $k$. In this fashion, the $l$-th-to-main diagonal of \eqref{eqn:cauchy:factor} will contain linear combinations of terms of the form
\[
\frac{1}{(\xi^i- z^i_j )^l}\qquad j=1, \dots, d. 
\]
Hence, thanks to \eqref{eqn:cauchymat}, $f(M)=0$ if and only if $f$ vanishes to its joint spectrum, and so do some linear combinations of its partial derivatives at the points of its joint spectrum. 
\begin{ex}
\label{ex:vanishing}
Let $\lambda$ and $\gamma$ be in $\D$, and set
\[
M=\left(\begin{bmatrix}
\lambda& 2\\
0& \lambda
\end{bmatrix} , \begin{bmatrix}
\gamma & 1\\
0 & \gamma
\end{bmatrix}\right).
\]
Then the joint spectrum of $M$ consists in just the point $z=(\lambda, \gamma)$ in $\D^2$ and, for any $f$ holomorphic in $\D^2$, 
\[
f(M)=\begin{bmatrix}
f(z) & 2\partial_1 f(z)+\partial_2f(z)\\
0 & f(z)
\end{bmatrix}.
\]
Thus $f(M)=0$ if and only both $f(z)=0$ and $2\frac{\partial f}{\partial z^1}(z)+\frac{\partial f}{\partial z^2}(z)=0.$
\end{ex}
Let $A=(A_n)_{n\in\N}$ be a sequence of commuting $d$-tuples with joint spectra in $\D^d$. For any positive integer $n$ and for any kernel $k$ in $\bound_d$, define
\begin{equation}
\label{eqn:hnd}
H_n^k:=\rkhs_k\ominus\{f\in\rkhs_k | f(A_n)=0\}\subseteq\rkhs_k. 
\end{equation}
By the above discussion $H^k_n$ is a finite dimensional sub-space of $\rkhs_k$ spanned by the kernels at the points of the joint spectrum of $A_n$, together with some linear combinations of kernel functions that represent partial derivatives at the joint spectrum of $A_n$. Moreover, if the matrices in $A_n$ have size $s_n$, then
\[
H^{k}_n=\{K_{A_n}(u, v)\,|\, u, v \in\C^{s_n}\},
\]
where, if $(e_j)_{j\in\N}$ is any orthonormal basis of $\rkhs_k$, $K_{A_n}(u, v)$ is defined by 
\begin{equation}
\label{eqn:dkernelmat}
K_{A_n}(u, v):=\sum_{j\in\N}\braket{v, e_j(A_n)u}_{\rkhs_k} e_j.
\end{equation}
In particular, the function $K_{A_n}(u, v)$ satisfies
\begin{equation}
\label{eqn:repd}
\begin{split}
\braket{f, K_{A_n}(u, v)}_{\rkhs_k}&=\sum_{j\in\N}\braket{e_j(A_n)u, v}_{\C^{s_n}}\braket{f, e_j}_{\rkhs_k}\\
&=\sum_{j\in\N}\braket{\braket{f, e_j}_{\rkhs_k}e_j(A_n) u, v}_{\C^{s_n}}\\
&=\left<\left(\sum_{j\in\N}\braket{f, e_j}_{\rkhs_k}e_j(A_n)\right)u, v\right>_{\C^{s_n}}\\
&=\braket{f(A_n)u, v}_{\C^{s_n}}\qquad f\in\rkhs_k,
\end{split}
\end{equation}
which says that \eqref{eqn:dkernelmat} does not depend on the choice of the basis $(e_n)_{n\in\N}$. Equation \eqref{eqn:repd} works as a \emph{reproducing property} for $K_{A_n}(u, v)$ and implies that $K_{A_n}(u, v)$ is linear in $v$ and conjugate-linear in $u$. Most importantly, \eqref{eqn:repd} implies that
\begin{equation}
\label{eqn:invariantd}
M_\phi^*(K_{A_n}(u, v))=K_{A_n}(u, \phi(A_n)^*v)\qquad \phi\in\h^\infty_d.
\end{equation}
It is worth noticing that, for any positive integer $n$, some of the $K_{A_n}(u, v)$ might repeat while varying $u$ and $v$. The following Lemma describes the relation on the pairs $(u_1, v_1)$ and $(u_2, v_2)$ in order for this to happen. If $l$ is a multi-index in $\N^d$, and $N$ is a $d$-tuple of commuting matrices, we will write
\[
N^l:=(N^1)^{l_1}\dots(N^d)^{l_d}.
\]
\begin{lemma}
\label{lemma:eqrel}
Let $k$ be a kernel in $\bound_d$, and let  $N$ be a $d$-tuple of $m\times m$ commuting matrices with spectra in $\D^d$. Let $u_1, v_1, u_2$ and $v_2$ be in $\C^m$. Then $K_N(u_1, v_1)=K_N(u_2, v_2)$ if and only if
\begin{equation}
\label{eqn:eqrel}
\braket{N^lu_1, v_1}_{\C^m}=\braket{N^lu_2, v_2}_{\C^m}\qquad l\in\N^d.
\end{equation}
\end{lemma} 
\begin{proof}
Let us first assume that $K_N(u_1, v_1)=K_N(u_2, v_2)$. Thus, since each monic monomial $z^l$ belongs to $\h^\infty_d$ and thanks to \eqref{eqn:invariantd},
\[
K_N(u_1, (N^l)^*v_1)=M^*_{z^l}(K_N(u_1, v_1))=M^*_{z^l}(K_N(u_2, v_2))=K_N(u_2, (N^l)^*v_2).
\]
Therefore, if $\phi$ is a polynomial such that $\phi(N)=Id$, the reproducing property in \eqref{eqn:repd} implies that
\[
\braket{u_1, (N^l)^*v_1}_{\C^m}=\braket{\phi, K_N(u_1, (N^l)^*v_1)}_{\rkhs_k}=\braket{\phi, K_N(u_2, (N^l)^*v_2)}_{\rkhs_k}=\braket{u_2, (N^l)^*v_2}_{\C^m}.
\]
Conversely, assume that \eqref{eqn:eqrel} holds. We have that $K_N(u_1, v_1)=K_N(u_2, v_2)$ if and only if their inner product with any $f$ in $\rkhs_k$ coincide:
\[
\braket{u_1, f(N)^*v_1}_{\C^m}=\braket{f, K_N(u_1, v_1)}_{\rkhs_k}=\braket{f, K_N(u_2, v_2)}_{\rkhs_k}=\braket{u_2, f(N)^*v_2}_{\C^m}\qquad f\in\rkhs_k.
\]
The rightmost and the leftmost hand side of the above equation coincide, thanks to \eqref{eqn:eqrel} and modulo writing $f$ in its power series.
\end{proof}
Observe that \eqref{eqn:eqrel} does {\bf not} depend on the kernel $k$ that we chose in $\bound_d$, but it depends exclusively on the algebraic structure of $M$.
\begin{ex}
Let $M$ be the pair of commuting matrices in Example \ref{ex:vanishing}. If $e_1$ and $e_2$ are the elements of the canonical basis of $\C^2$, then 
\[
\begin{split}
&K_M(e_1, e_1)=K_M(e_2, e_2)=k_z,\\
&K_M(e_1, e_2)=2k^{(1, 0)}_z+k^{(0,1)}_z\\
&K_M(e_2, e_1)=0,
\end{split}
\]
where 
\[
k_z^{(1, 0)}:=\frac{\partial k_z}{\partial{\overline{z}^1}}
\]
and
\[
k_z^{(0, 1)}:=\frac{\partial k_z}{\partial{\overline{z}^2}}
\]
are the kernels in $\rkhs_k$ that represent, respectively, one derivative in the first variable and one derivative in the second variable.
\end{ex}
In Section \ref{sec2}, we will extend Theorem \ref{theo:bidisk} to pairs of commuting matrices by replacing normalized admissible kernels with the  collections $(H^{k}_n)_{n\in\N}$:
\begin{theo}
\label{theo:bidiskmat}
A sequence of pairs of commuting matrices is interpolating if and only if there exists a positive $C$ such that, for any admissible kernel $k$ on the bi-disc, the sequence
\[
H^k=(H^{k}_n)_{n\in\N}
\]
is a Riesz system in $\rkhs_k$ with Riesz bound $C$. 
\end{theo}
 Since the approach we use to study interpolating sequences of matrices is mainly defined by looking at the Euclidean geometry of reproducing kernel Hilbert spaces, and since we do not know whether $\A_d=\bound_d$ for $d\geq3$, we can only give a (stronger) sufficient condition to replace \eqref{eqn:ss} in the case of the bi-disc:
\begin{theo}
\label{theo:bernmat}
Let $A=(A_n)_{n\in\N}$ be a sequence of pairs of commuting matrices with spectra in $\D^2$ such that
\begin{equation}
\label{eqn:suffd}
\prod_{n\in\N}\rho_G(A_n, A_{\N\setminus n})>0.
\end{equation}
Then $A$ is interpolating.
\end{theo}
Let us observe that \eqref{eqn:suffd} is a rather strong separation condition on $A=(A_n)_{n\in N}$. In order to extend the sufficient condition in Theorem \ref{theo:poly:interp}, one has would have to give a positive answer to the following question: 
\begin{question}
Is $A$ interpolating, provided that
\[
\inf_{n\in\N}\prod_{j\ne n}\rho_G(A_n, A_j)>0?
\]
\end{question}
This remain, for us, open. \\
Section \ref{sec:comex} will show that, given any sequence $(m_n)_{n\in\N}$ in $\N$ and any $d\geq2$, there exists an interpolating sequence $A$ of $d$-tuples of commuting matrices such that each $A_n$ has  $m_n\times m_n$ coordinates.
\subsection{Interpolating $d$-tuples of Non-Commuting Matrices }
In order to study interpolating properties of a sequence of eventually non-commuting $d$-tuples of matrices, one has to change the class of functions to apply to such a sequence. Specifically, the robust and highly active field of {\bf non-commutative (NC) functions} is the environment in which the discussion of Section \ref{sec3} will take place. \\
Fix $d$ in $\N\cup\{\infty\}$. For any positive integer $n$, let $\Mat^1_n$ be the set of all $n\times n$ matrices with coefficient in $\C$ and, more generally, let $\Mat^d_n$ be the set of all $d$-tuples of square matrices of size $n$. If $d=\infty$, we require that the row norm
\begin{equation}
\label{eqn:ncnorm}
||X||:=\left|\left|\sum_{i}X^i(X^i)^*\right|\right|
\end{equation}
is bounded for any $X$ in $\Mat^\infty_n$. Each $\Mat^d_n$ is normed with \eqref{eqn:ncnorm} and endowed with the induced topology. Let us consider now arbitrarily large sizes by defining
\[
\Mat_d:=\bigcup_{n=1}^\infty\Mat^d_n
\] 
as the disjoint union in $n$ of $d$-tuples of $n\times n$ matrices. A topology of interest for us is the so-called \emph{disjoint union (DU) topology}: a subset $\Omega$ of $\Mat_d$ is open if and only if all of its $n$ components
\[
\Omega(n):=\Omega\cap\Mat^d_n\qquad n\in\N
\]
are open in $\Mat^d_n$. Moreover, $\Omega$ is a {\bf non-commutative (NC) set} if it is closed under direct sums:
\[
Z, W\in\Omega\implies Z\oplus W\in\Omega.
\]
A class of functions that can be defined on such an $\Omega$ is $\C\braket{z_1, \dots, z_d}$, the set of all \emph{free polynomials} in $d$ non-commuting variables. More generally, a function $f\colon\Omega\to\Mat_1$ is a {\bf non-commutative (NC) function} if
\begin{itemize}
\item $f$ is \emph{graded}: if $Z$ is in $\Omega(n)$, then $f(Z)$ belongs to $\Mat^1_n$;
\item $f$ \emph{respects direct sums}: for any $Z$ and $W$ in $\Omega$, then $f(Z\oplus W)=f(Z)\oplus f(W)$;
\item $f$ \emph{respects similarities}: for any $Z$ in $\Omega(n)$ and for any invertible $P$ in $\Mat^1_n$, then $f(P^{-1}ZP)=P^{-1}f(Z)P$, provided that $P^{-1}ZP$ belongs to $\Omega$.
\end{itemize}
One can easily check that any NC polynomial is an NC function on $\Omega$. An NC function $f$ is \emph{holomorphic} on $\Omega$ if it is locally bounded. 
\\
The domain of interest for us in order to extend interpolation results to this NC setting is the {\bf NC unit ball}
\[
\Ball_d:=\left\{ Z\in\Mat_d\,\big|\,||Z||<1\right\},
\]
and the interpolating functions will be chosen from the algebra of {\bf bounded NC analytic functions}
\[
\rkhs^\infty_d:=\{f\quad\text{NC holomorphic on}\,\Ball_d\quad|\quad ||f||_\infty<\infty\},
\]
where
\[
||f||_\infty:=\sup_{Z\in\Ball_d}||f(Z)||.
\]
Definition \ref{defi:intmat} adapts also to this non-commutative setting: a sequence $(Z_n)_{n\in\N}$ in $\Ball_d$ (of eventually non-constant dimensions) is {\bf interpolating} for $\rkhs^\infty_d$ if for any bounded sequence $(\phi_n)_{n\in\N}$ in $\rkhs^\infty_d$ there exists a function $\phi$ in $\rkhs^\infty_d$ such that 
\[
\phi(Z_n)=\phi_n(Z_n)\qquad n\in\N.
\]
As for the well-known scalar case, the understanding of NC interpolating sequences goes through the study of related Hilbert spaces. Let $\W_d$ be the set of words with $d$ generators, and let
\[
\rkhs^2_d:=\left\{f=\sum_{l\in\W_d}a_lZ^l\quad\bigg|\quad||f||_2^2:=\sum_{l\in\W_d}|a_l|^2<\infty\right\}
\] 
be the so-called \emph{NC Drury Arveson space} of those formal NC power series with square-summable coefficients. It turns out that $\rkhs^2_d$ can be seen as an {\bf NC reproducing kernel Hilbert space}. More specifically, for any $W$ in $\Ball_d(n)$ and for any pairs of vectors $u$ and $v$ in $\C^n$, let
\begin{equation}
\label{eqn:ncrkhs}
K_W(u, v)(Z):=\sum_{l\in\W_d}\braket{v, W^lu}Z^l\qquad Z\in\Ball.
\end{equation}
The function $K_W(u,v)$ is linear in $v$ and conjugate-linear in $u$. Salomon, Shalit and Shamovich showed in \cite[Prop. 3.2]{orr} that $\rkhs^2_d$ is generated by the NC functions in \eqref{eqn:ncrkhs}, and that for any $f$ in $\rkhs_d^2$ and for any $W$ in $\Ball_d(n)$
\begin{equation}
\label{eqn:ncrep}
\braket{f, K_W(u, z)}_{\rkhs^2_d}=\braket{f(W)u, v}_{\C^n}.
\end{equation}
Moreover, the multiplier algebra of $\rkhs_d^2$
\[
\M_{\rkhs^2_d}:=\{\phi\,\,\text{NC holomorphic function on}\,\Ball_d\,|\, \phi f\in\rkhs_d^2, f\in\rkhs_d^2 \}
\]
is isometrically identifiable with $\rkhs^\infty_d$, \cite[Coro. 3.6]{orr}, and the reproducing property in \eqref{eqn:ncrep} implies that 
\[
M_\phi^*(K_W(u, v))=K_W(u, \phi(W)^*v)\qquad \phi\in\rkhs^\infty_d.
\]
Similarly to Section \ref{sec2}, in Section \ref{sec3} we will describe interpolating sequences in terms of \emph{separated sequences}. Moreover, we will see how to separated elements in $\Ball_d$ correspond separated sub-spaces of $\rkhs^2_d$. Specifically, the \emph{NC Gleason distance} between two elements $Z$ and $W$ of $\Ball_d$ is the inverse of the least norm of an $\rkhs^\infty_d$ NC function that vanishes at $W$ and that is the identity at $Z$:
\[
\rho_{NC}(Z, W):=\frac{1}{\inf\{||\phi||_\infty\,|\, \phi(Z)=Id, \phi(W)=0\}}.
\]
As in \eqref{eqn:dsepseq}, we might also want to separate $Z$ from a whole sequence $(W_n)_{n\in\N}$ via
\[
\rho_{NC}(Z, (W_n)_{n\in\N}):=\frac{1}{\inf\{||\phi||_\infty\,|\, \phi(Z)=Id, \phi(W_n)=0, n\in\N\}}.
\]
We will see in Section \ref{sec3}  how an analogue of the Pick property for the NC Drury-Arveson space implies that separation conditions on a sequence $(Z_n)_{n\in\N}$ in $\bound_d$ correspond to separated sequences of sub-spaces of $\rkhs^2_d$ defined as
\begin{equation}
\label{eqn:ncmodel}
\rkhs_n:=\rkhs^2_d\ominus\{f\in\rkhs^2_d\,|\, f(Z_n)=0\}.
\end{equation}
Since the right hand side of the reproducing property \eqref{eqn:ncrep} vanishes for any $u$ and $v$ if and only if $f(W)=0$, we have that, 
\begin{equation}
\label{eqn:nckernel}
\rkhs_n=\{K_{Z_n}(u, v)\, |\, u, v\in\C^{m_n}\},
\end{equation}
where we assume that $Z_n$ belongs to $\Ball_d(m_n)$ for any $n$.\\
Interpolating sequences can be then characterized via separation conditions in $\rkhs^2_d$:
\begin{theo}
\label{theo:ncriesz}
$Z$ is interpolating if and only if the sequence $(\rkhs_n)_{n\in\N}$ is a Riesz system in $\rkhs^2_d$.
\end{theo}
As a consequence, we will show that a separation condition on the sequence  $Z$ is sufficient for it to be interpolating:
\begin{theo}
\label{theo:ssnc}
If 
\begin{equation}
\label{eqn:ncss}
\prod_{n\in\N}\rho_{NC}(Z_n, (Z_j)_{j\ne n})>0
\end{equation}
then $Z$ is interpolating.
\end{theo}
Condition \eqref{eqn:ncss} is a rather strong separation condition on the sequence $Z$. It would be interesting for us to know whether such condition can be relaxed:
\begin{question}
\label{q:ncbern}
Is any sequence $Z$ in $\Ball_d$ satisfying
\[
\inf_{n\in\N}\prod_{j\ne n}\rho_{NC}(Z_n, Z_j)>0
\]
an interpolating sequence?
\end{question}
A positive answer to Question \ref{q:ncbern} would extend a result from Berndtsson in \cite{bernball} to this non commutative setting.\\
An even more ambitious goal would be to prove an analogue of Carleson interpolation Theorem for the noncommutative setting. In particular, we ask whether strongly separated sequences are interpolating, and whether weakly separated sequences whose sequence $(\rkhs_n)_{n\in\N}$ is a Bessel system are interpolating:
\begin{question}
Is $Z$ interpolating, provided that
\[
\inf_{n\in\N}\rho_{NC}(Z_n, (Z_j)_{j\ne n})>0?
\]
\end{question}
\begin{question}
Is $Z$ interpolating, provided that 
\[
\inf_{n\ne j}\rho_{NC}(Z_n, Z_j)>0
\]
and $(\rkhs_n)_{n\in\N}$ is a Bessel system?
\end{question}
An example of  an interpolating sequence of pairs of non-commuting matrices in $\bound_2$ will be given in Section \ref{sec:ncex}.\\

The author would like to thanks John McCarthy for the many discussions that led to this note.
\section{Commuting Pairs}
\label{sec2}
The main goal of this Section is to prove Theorem \ref{theo:bidiskmat} and Theorem \ref{theo:bernmat}.\\
In order to do so, we will first extend in Section \ref{sec:agler} Agler's argument in \cite{aglerpp} to a commuting matrix nodes interpolation problem.\\
 Section \ref{sec:diagonal} will then prove Theorem \ref{theo:diagonal}, which contains Theorem \ref{theo:bidiskmat} and also shows that Definition \ref{defi:dintmat} can be relaxed to an a priori weaker (yet equivalent) notion of interpolating $d$-tuples of commuting matrices where a target sequence can be identified with a bounded sequence in $\C$. \\
In Section \ref{sec:bernmat} we will then look for a sufficient condition in order for a sequence of $d$-tuples of commuting matrices to be interpolating which is stated in terms of the Gleason distance on the polydisc, and thus prove Theorem \ref{theo:bernmat}.
\subsection{Admissible Kernels}
\label{sec:agler}
The key step for the proof of Theorem \ref{theo:bidiskmat} is to extend Agler's argument in \cite{aglerpp} to the matrix case. Let $A_1, \dots, A_N$ be finitely many $d$-tuples of commuting matrices with spectra in $\D^d$, of eventually different sizes $s_1, \dots, s_N$. Given the formal sets
\[
K_j:=\{K_j(u, v) | u, v \in\C^{s_j} \}\qquad j=1, \dots, N
\]
define an equivalence relation $\sim_j$ on $K_j$ that makes the elements in $K_j$ linear in $v$, conjugate linear in $u$, and that sets $K_j(u_1, v_1)\sim_jK_j(u_2, v_2)$ if and only if 
\[
\braket{A_j^lu_1, v_1}_{\C^{s_j}}=\braket{A_j^ lu_2, v_2}_{\C^{s_j}}\qquad l\in\N^d.
\]
Let
\[
F_j:=K_j/\sim_j\qquad j=1,\dots, N
\]
and let $F$ be the set of formal linear combinations of the elements of $\cup_{j=1}^NF_j$. Since each $F_j$ is a finite dimensional vector space, so is $F$. A \emph{kernel} $\tilde{k}$ on $F$ is a choice of a strictly positive scalar product on $F$. Let $F_{\tilde{k}}$ be the Hilbert space $(F, \tilde{k})$.  In this finite-nodes setting, $\tilde{k}$ is {\bf admissible} if, for any $i=1, \dots, d$, the linear map $T^{\tilde{k}}_i\colon F_{\tilde{k}}\to F_{\tilde{k}}$ defined by
\begin{equation}
\label{eqn:tki}
T^{\tilde{k}}_i([K_j(u, v)]_{\sim_j}):=[K_j(u, (A_j^i)^*v)]_{\sim_j}
\end{equation}
is a contraction. Observe that each $T^{\tilde{k}}_i$ is well defined thanks to Lemma \ref{lemma:eqrel}. In particular, if $k$ is an admissible kernel on $\D^d$, then
\[
\spa\{H^k_{A_j}\,|\, j=1\dots, N\}\subset\rkhs_k
\]
is isometric to some $F_{\tilde{k}}$. In other words, $F_{\tilde{k}}$ works as a \emph{restriction} of an admissible kernel to the finite nodes $A_1, \dots, A_N$. For the sake of notation, $k$ will indicate both a kernel on $\D^d$ and the scalar product on $F$ which restricts $k$ to our finite nodes of $d$-tuples of matrices.\\
 Let $\phi_1,\dots, \phi_N$ be targets in $\h^\infty_2$, and suppose that there exists a contraction $\phi$ in $\h^\infty_2$ such that
\[
\phi(A_j)=\phi_j(A_j)\qquad j=1,\dots, N.
\]
Then, thanks to \eqref{eqn:invariantd}, for any kernel in $k$ in $\A_2$ the map $R^k\colon\spa\{H^k_{A_j}\,|\, j=1, \dots, N\}\to\spa\{H^k_{A_j}\,|\, j=1, \dots, N\}$ given by
\begin{equation}
\label{eqn:rk}
R^k(K_{A_j}(u, v))=K_{A_j}(u, \phi_j(A_j)^*v)\qquad j=1, \dots, N
\end{equation}
is a restriction of $M^*_\phi$, and therefore is a contraction. Conversely,
\begin{theo}
\label{theo:extension}
Let $\phi_1, \dots, \phi_N$ be in $\h^\infty_2$ such that, if $R^k$ is defined as in \eqref{eqn:rk},
\begin{equation}
\label{eqn:unifbound}
\sup_{k\in\A_2}||R^k||\leq 1.
\end{equation}
Then there exists a function $\phi$ in $\h^\infty_2$ such that $||\phi||_\infty\leq1$ and
\[
\phi(A_j)=\phi_j(A_j)\qquad j=1, \dots, N.
\]
\end{theo}
Theorem \ref{theo:extension} works as an extension property for all admissible kernels, provided that the necessary bounds are uniform in $\A_2$. \\
The proof of Theorem \ref{theo:extension} has a strong operator theory flavour: the first main tool is the following Lemma, due to Parrot:
\begin{lemma}[Parrot's Lemma]
\label{lemma:parrot}
Let $\rkhs_i$ and $\mathcal{K}_i$, $i=1, 2$, be Hilbert spaces, Let $A\colon\rkhs_1\to\mathcal{K}_1$, $B\colon\rkhs_2\to\mathcal{K}_1$ and $C\colon\rkhs_1\to\mathcal{K}_2$ be linear operators. For any $D\colon\rkhs_2\to\mathcal{K}_2$, let
\[
W_D:=\begin{bmatrix}
A & B\\
C & D
\end{bmatrix}\colon\rkhs_1\oplus\rkhs_2\to\mathcal{K}_1\oplus\mathcal{K}_2.
\]
Then 
\[
\sup_{D\in\mathcal{B}(\rkhs_2, \mathcal{K}_2)}||W_D||=\max\left\{\left|\left|\begin{bmatrix}
A\\
C
\end{bmatrix}\right|\right|_{\mathcal{B}(\rkhs_1, \mathcal{K}_1\oplus\mathcal{K}_2)}\,,\, \left|\left|\begin{bmatrix}
A & B
\end{bmatrix}\right|\right|_{\mathcal{B}(\rkhs_1\oplus\rkhs_2, \mathcal{K}_1)}\right\}.
\]
\end{lemma}
For a proof, see \cite[Lemma B.1]{john}.\\
The second technical tool can be proved for any $d\geq2$, and it will allow us to go from one admissible kernel to another, making a full use of the uniform bound in \eqref{eqn:unifbound}. Let $k$ be an admissibile kernel on $\D^d$, and let $z$ be a point in $\D^d$. Let 
\[
H_0:\spa\{H^k_{A_j}\,|\, j=1, \dots, N\}\subset\rkhs_k
\]
and $H_1:=\spa\{H_0, k_z\}$. Assume that $k_z$ is not in $H_0$, so that $H_0$ is strictly contained in $H_1$. Define $\mathcal{G}_0:=H_1\ominus\spa\{k_z\}$, and let $L_0\colon H_0\to \mathcal{G}_0$ be the restriction on $H_0$ of the orthogonal projection $P_0$ onto $\mathcal{G}_0$:
\[
L_0:=(P_0)_{| H_0}.
\]
 As we already pointed out, $H_0$ is isometric to a finite admissible kernel structure $F_k$ on $A_1, \dots, A_N$. It turns out that $L(H_0)$ corresponds to a finite admissible kernel structure as well:
\begin{lemma}
\label{lemma:proj}
Define, for any $j=1, \dots, N$ and for any $u$ and $v$ in $\C^{s_j}$
\[
G_j(u, v):=L_0(K_{A_j}(u, v))\in \mathcal{G}_0. 
\]
Then the vector space
\[
\{G_j(u, v)\,|\, j=1, \dots, N, u, v\in\C^{s_j}\}\subset\rkhs_k
\] 
together with the inner product 
\[
g(G_j(u_1, v_1), G_l(u_2, v_2)):=\braket{G_{j}(u_1, v_1), G_{l}(u_2, v_2)}_{\rkhs_k}
\]
is a finite admissible kernel structure on $A_1, \dots, A_N$.
\end{lemma}
\begin{proof}
Observe that $H_1$ is an admissible kernel structure on the $N+1$ nodes, $A_1, \dots, A_N, z$. Hence the map $\tilde{T}^k_i$ such that 
\[
\tilde{T}^k_i(x)=\begin{cases}
T^k_i(x)\quad\text{if}&\quad x\in H_0\\
\overline{z^i}~k_z\quad\text{if}&\quad x=k_z
\end{cases}\qquad i=1, \dots, d
\]
extends $T^k_i$ defined in \eqref{eqn:tki} to a contraction on $H_1$. Since the maps 
\[
S_i:=P_0(\tilde{T}^k_i)_{|\mathcal{G}_0}\qquad i=1,\dots, d
\]
are contractions, it suffices to show that 
\begin{equation}
\label{eqn:multvarg}
T^g_i=S_i\qquad i=1, \dots, d.
\end{equation}
Since $\spa\{k_z\}$ is invariant under each $\tilde{T}^k_i$, we have
\[
P_0\tilde{T}^k_i(Id-P_0)=0\qquad i=1, \dots, d,
\]
and therefore, for any $f$ in $H_0$
\[
S_iL(f)=P_0\tilde{T}^k_iP_0(f)=P_0\tilde{T}^k_i(f)=P_0T^k_i(f)=LT^k_i(f)\qquad i=1, \dots, d.
\]
Thus $S_i=L T^k_i L^{-1}$ is similar to $T^k_i$ and therefore 
\[
S_i(G_j(u, v))=L(K_j(u, (A^i)^*v))=G_j(u, (A^i)^*v)=T^g_i(G_j(u, v)),
\]
proving \eqref{eqn:multvarg}.
\end{proof}
We are now ready to prove Theorem \ref{theo:extension}. In order to do so, it suffices to show that, for any $z$ in $\D^d$, one can choose an optimal $w$ in $\C$ such that, for any admissible kernel $k$, $\tilde{R}^k\colon H_1\to H_1$ given by 
\begin{equation}
\label{eqn:rktilda}
\tilde{R}^k_w(x):=\begin{cases}
R^k(x)\quad\text{if}&\quad x\in H_0\\
\overline{w}k_z\quad\text{if}&\quad x=k_z
\end{cases}
\end{equation}
extends $R^k$ defined in \eqref{eqn:rk} without changing its norm:
\begin{theo}
\label{theo:pointext}
For any $z$ in $\D^d$, there exists a $w$ in $\C$ such that 
\[
\sup_{k\in\A_d}||\tilde{R}^k_w||=\sup_{k\in\A_d}||R^k||.
\]
\end{theo}
Theorem \ref{theo:extension} follows then from Theorem \ref{theo:pointext}: let $Z=(z_n)_{n\in\N}$ be a \emph{sequence of uniqueness} for $\D^2$, i. e., a sequence such that any holomorphic function is uniquely determined by its values at $Z$ (for example, any sequence converging to $0$). By iterating  Theorem \ref{theo:pointext} we extend $R^k$ isometrically to the nested sub-spaces
\[
H^k_n:=\spa\{H^k_{n-1}, k_{z_n}\}\qquad n\in\N
\]
by choosing an optimal value $w_i$ at each step, independently of the admissible kernel $k$. This leads to construction of a map on the Hardy space $H^2(\D^2)$
\[
R\colon\h^2(\D^2)\to\h^2(\D^2)
\]
that has all the Szegö kernels as its eigenvetors, and which hence commutes with any $M_{z^i}^*$. Such a map will have then to be the adjoint of multiplication by
\[
\phi:=R^*(1),
\]
and since $M^*_\phi=R$ is a contraction, then $||\phi||_\infty=1$. Moreover, since $M^*_\phi$ conicide with $M^*_{\phi_j}$ on each $H^k_j$, then $\phi$ agrees with $\phi_j$ on each pair $A_j$. 
\begin{proof}[Proof of Theorem \ref{theo:pointext}]
Let $k$ be an admissible kernel on $\D^d$, and let $R^k$ be defined as in \eqref{eqn:rk}. Let $\tilde{R}^k_w$ be the extension in \eqref{eqn:rktilda}. Define
\[
\begin{split}
&\mathcal{G}:=H_1\ominus\spa\{k_z\}\\
&\mathcal{F}:=H_1\ominus H_0.
\end{split}
\]
Split then $\tilde{R}^k_w$ into
\begin{equation}
\label{eqn:split}
\tilde{R}^k_w=\begin{bmatrix}
A & B\\
C & D
\end{bmatrix},
\end{equation}
where
\[
\begin{split}
&A\colon H_0\to\mathcal{G}\\
&B\colon\mathcal{F}\to\mathcal{G}\\
&C\colon H_0\to\spa\{k_z\}\\
&D\colon\mathcal{F}\to\spa\{k_z\}.
\end{split}
\]
The column $\begin{bmatrix}
A\\
C
\end{bmatrix}$ is then $R^k$, while the top row is $\tilde{R}^k_w$ pre-composed with the orthogonal projection $P$ onto $\mathcal{G}$:
\[
\begin{bmatrix}
A & B
\end{bmatrix}=P\tilde{R}^k_w.
\]
In particular, $\begin{bmatrix}
A\\
C
\end{bmatrix}$ does not depend on $w$. Most importantly, $B$ does not depend on $w$ either, since $\tilde{R}^k_w$ has $k_z$ as one of its eigevenctors and $\mathcal{G}$ is orthogonal to $k_z$, which in particular implies
\[
B=(P\tilde{R}^k_w)_{|\mathcal{F}}=(PR^k)_{|\mathcal{F}}.
\]
Therefore \emph{the whole dependence on $w$ in \eqref{eqn:split} is carried by $D$}. In particular, 
\[
D\colon f\in\mathcal{F}\mapsto \frac{\overline{w}||f||^2}{\braket{k_z, f}_{\rkhs_k}} k_z
\]
ranges among any possible complex number, as $w$ ranges in $\C$. Moreover, observe that if 
\[
T^k:=(P\tilde{R}^k_w)_{|\mathcal{G}},
\] 
then $||P\tilde{R}^k_w||\leq||T^k||$ (and hence the two norms are indeed equal), since for any $v=u+\xi$ in $H_1$, where $u$ is in $\mathcal{G}$ and $\xi$ is in $\spa\{k_z\}$, one has
\[
||P\tilde{R}^k_w(v)||=||T^k(u)||\leq||T^k||||u||\leq||T^k||||v||,
\]
thanks to orthogonality. Thus Lemma \ref{lemma:parrot} implies that there exists a choice of $w$ such that
\[
||\tilde{R}^k_w||=\max\{||R^k||, ||T^k|| \}.
\]
To conclude, it suffices to observe that thanks to Lemma \ref{lemma:proj} $\mathcal{G}$ carries an admissible structure $g$, and that therefore
\[
T^k=p(T_1^k,\dots, T_d^k)
\] 
is unitarly equivalent to
\[
R^g=p(T_1^g,\dots, T_d^g),
\]
for any polynomial $p$ such that $p(A_j)=\phi_j(A_j)$ for any  $j=1, \dots, N$. Thus
\[
||\tilde{R}^k_w||=\max\{||R^k||, ||R^g|| \}\leq\sup_{g\in\A_d}||R^g||.
\]
Since $k$ was chosen arbitrarily in $\A_d$, this concludes the proof.
\end{proof}
\begin{rem}
\label{rem:euclidsepd}
As a main consequence of Theorem \ref{theo:extension}, the Gleason distance $\rho_G(M, N)$ between two pairs of commuting matrices corresponds to the sine of the least angle between $H_M^k$ and $H_N^k$, where $k$ ranges among all admissible kernels in $\D^2$. Indeed, if $\phi$ in $\h^\infty_2$ has norm $C$ and separates $M$ and $N$, i.e.
\[
\phi(M)=Id\qquad\phi(N)=0,
\]
then for any admissible kernel $k$ the operator $M^*_\phi$ acts like the identity on $H^k_M$ and like the zero operator on $H^k_N$. Since $||M^*_\phi||=C$, the angle between $H^k_M$ and $H^k_N$ is greater than $1/C$.\\
Conversely, if
\[
\inf_{k\in\A_2}\sin(H^k_M, H^k_N)\geq\frac{1}{C}, 
\]
then Theorem \ref{theo:extension} ($A_1=M, A_2=N, \phi_1=1, \phi_2=0$) implies that there exists a function $\phi$ whose $\h^\infty$ norm does not exceed $C$ that separates $M$ and $N$, hence $\rho_G(M, N)\geq1/C$.
\end{rem}
\subsection{Diagonal Targets}
\label{sec:diagonal}
Theorem \ref{theo:bidiskmat} follows from Theorem \ref{theo:extension}:
\begin{theo}
\label{theo:diagonal}
Let $A=(A_n)_{n\in\N}$ be a sequence of pairs of commuting matrices with spectra in the bi-disc, and let $H^k=(H_n^k)_{n\in\N}$ be, for any kernel in $\A_2$, the associated sequence of closed sub-spaces of $\rkhs_k$ defined in \eqref{eqn:hnd}. The following are equivalent:
\begin{description}
\item[(i)] $A$ is interpolating;
\item[(ii)] For any bounded sequence $(w_n)_{n\in\N}$ in $\C$ there exists a function $\phi$ in $\h^\infty_2$ such that
\[
\phi(A_n)=w_n~Id\qquad n\in\N;
\]
\item[(iii)] There exists a $C>0$ such that, for any $k$ in $\A_2$, $H^k$ is a Riesz system in $\rkhs_k$ with Riesz bound $C$;
\item[(iv)] There exists a sequence $(f_n)_{n\in\N}$ such that $f_n(A_j)=\delta_{n, j}$ and 
\[
\sup_{z\in\D^2}\sum_{n\in\N}|f_n(z)|<\infty.
\]
\end{description}
\end{theo}
The proof of the equivalence between (ii) and (iv) follows the same outline of \cite[Th. 4.1]{dayoss}: such an argument is indeed valid since Montel's Theorem extends to $\h^\infty_2$. Therefore, it suffices to show (i)$\implies$(ii)$\implies$(iii)$\implies$(i). Moreover, the equivalence between (i) and (ii) says that, as for the one variable case, the notion of interpolating sequences for pairs of commuting matrices is indeed equivalent to one that a priori is weaker, asking to find interpolating functions for only diagonal targets. 
\begin{proof}
Trivially, (i) implies (ii), by considering a constant target sequence 
\[
\phi_n(z)=w_n\qquad z\in\D^2.
\]
The implication (ii)$\implies$(iii) follows from Theorem \ref{theo:nikriesz}, by observing that, thanks to (ii), for any admissible kernel $k$ on $\D^2$
\[
\mu^k=M^*_\phi\colon\underset{n\in\N}{\spa}\{H_n^k\}\to\underset{n\in\N}{\spa}\{H^k_n\} 
\]
whenever
\[
(\mu^k)_{H^k_n}=w_n~Id_{|H^k_n}.
\]
Finally, (iii)$\implies$(i) follows from Theorem \ref{theo:extension}, Theorem \ref{theo:nikriesz} and a normal family argument. Indeed, let $(\phi_n)_{n\in\N}$ be a bounded sequence in $\h^\infty_2$, and fix a positive integer $N$. By $H^k$ having a Riesz bound $C$ for any admissible $k$, one has that the operator $R^k$ defined in \eqref{eqn:rk} is bounded by $C$ for any $k$ in $\A_2$, thanks to Theorem \ref{theo:nikriesz}. Thus by Theorem \ref{theo:extension} there exists a function $f_N$ whose $\h^\infty_2$ norm doesn't exceed $C$ and 
\[
f_N(A_1)=\phi_1(A_1),\dots, f_N(A_N)=\phi_N(A_N).
\]
Since the sequence $(f_N)_{N\in\N}$ is bounded in $\h^\infty_2$, it has a sub-sequence that converges in $\h^\infty_2$ to $\phi$, which agrees with $\phi_n$ at $A_n$, for any $n$ in $\N$.
\end{proof}
\subsection{A More Explicit Sufficient Condition}
\label{sec:bernmat}
In this Section we will prove Theorem \ref{theo:bernmat}. Thanks to Remark \ref{rem:euclidsepd}, condition \eqref{eqn:suffd} implies that  
\begin{equation}
\label{eqn:suffdh}
\prod_{n\in\N}\sin\left(H^k_n, \underset{j\ne n}{\overline{\spa}}\{H_j^k\}\right)
\end{equation}
is uniformly bounded below for any $k$ in $A_2$. Thanks to Theorem \ref{theo:bidiskmat}, it suffices then to show that any sequence $H=(H_n)_{n\in\N}$ of closed subspaces of a Hilbert space $\rkhs$ satisfying \eqref{eqn:suffdh} is a Riesz system. In order to do so, we will use Theorem \ref{theo:nikriesz}, (i)$\iff$(iv). Fix a finite subset $\sigma$ of $\N$, and enumerate the elements in $\sigma^c$ by $\{j_1, j_2, \dots\}$. For the sake of brevity, let $H_\sigma:=\underset{j\in\sigma}{\overline{\spa}}\{H_j\}$ and
\[
S_i:=\spa\{H_\sigma, H_{j_1}, \dots, H_{j_i}\}\qquad i\in\N.
\]
Namely, if $S_0=H_\sigma$, then each $S_i$ is obtained by adding a subspace not labeled in $\sigma$ to the linear span of $S_{i-1}$. Let $P_0$ be the identity on $H_\sigma$, and define $P_i\colon S_i\to S_i$ by
\[
P_i(x):=\begin{cases}
P_{i-1}(x)\quad&\text{if}\quad x\in S_{i-1}\\
0\quad&\text{if}\quad x\in H_{j_i}
\end{cases}.
\]
Thanks to Theorem \ref{theo:nikriesz}, (i)$\iff$(iv), we need to prove the following
\begin{prop}
\label{prop:ssriesz}
Let $(H_n)_{n\in\N}$ be a sequence of closed sub-spaces of a Hilbert space $\rkhs$ such that
\begin{equation}
\label{eqn:sshs}
\prod_{n\in\N}\sin(H_ n, H_{\N\setminus\{n\}})>0
\end{equation}
Then 
\begin{equation}
\label{eqn:supnorms}
\sup_{\sigma\,\text{finite}}\lim_{i\to\infty}||P_i||<\infty.
\end{equation}
\end{prop}
This will be done by analyzing how much does the operator norm increase when extending $P_i$ to $P_{i+1}$:
\begin{lemma}
\label{lemma:hilbertriesz}
Let $H$ and $F$ be two closed subspaces of a Hilbert space $\mathcal{H}$ that intersect trivially. Let $d$ be the distance (equivalently, the sine of the angle) between $H$ and $K$, and let $T$ be any bounded operator from $H$ to itself. Then the operator $\tilde{T}\colon \spa\{ H, F\}\to\spa\{ H, F\}$ such that
\[
\tilde{T}(x):=\begin{cases}
T(x)\quad\text{if}&\quad x\in H\\
0\quad\text{if}&\quad x\in F
\end{cases}
\]
extends $T$ and has norm
\[
||\tilde{T}||\leq\frac{1}{d}||T||.
\] 
\end{lemma}
\begin{proof}
Let $G:=\spa\{H, F\}\ominus H$ be the orthogonal complement of $H$ in $\spa\{H, F\}$, and fix a unit vector $x$ in $\spa\{H, F\}$. Then $x$ can be written uniquely as 
\[
x=\alpha y+\beta z\qquad y\in G\quad z\in H,
\]
where $y$ and $z$ are unit vectors and $|\alpha|^2+|\beta|^2=1$, thanks to orthogonality. Moreover, $y$ can be written as 
\[
y=h+f\qquad h\in H\quad f\in F,
\]
where
\[
||h||^2=\frac{1-\sin^2(f, H)}{\sin^2(f, H)}\leq\frac{1-d^2}{d^2}
\]
and $\tilde{T}(y)=\tilde{T}(h)=T(h)$. Therefore, if $r:=\sqrt{1-d^2}/d$, we have
\[
\begin{split}
||\tilde{T}(x)||^2=&||T(\alpha h+\beta z)||^2\\
\leq&||T||^2~||\alpha h+\beta z||^2\\
=&||T||^2~(\alpha^2 r^2+\beta^2+2\alpha\overline{\beta}\mathrm{Re}\braket{h, z})\\
\leq&||T||^2~(\alpha^2 r^2+\beta^2+2|\alpha||\beta|r)\\
=&||T||^2(r|\alpha|+|\beta|)^2.
\end{split}
\]
Thus
\[
||\tilde{T}||\leq||T||~\sup_{|\alpha|^2+|\beta|^2=1}(r|\alpha|+|\beta|)=||T||~\sqrt{1+r^2}=\frac{||T||}{d}.
\]
\end{proof}
In particular,
\[
||P_n||^2\leq\prod_{i=1}^n\frac{1}{\sin^2(H_{j_i}, S_i)}\leq\prod_{i=1}^n\frac{1}{\sin^2(H_ i, H_{\N\setminus\{i\}})}
\]
is uniformly bounded in $n$ (and $\sigma$) if 
\[
\prod_{i=1}^n \sin(H_ i, H_{\N\setminus\{i\}})
\]
is bounded below, and Proposition \ref{prop:ssriesz} follows. This concludes the proof of Theorem \ref{theo:bernmat}.
\subsection{An Example via Random Interpolating Sequences in the Polydisc}
\label{sec:comex}
In certain instances, considering a random sequence (of scalars) in the polydisc can help to understand the conditions in Theorem \ref{theo:poly:interp} and their relation. This is the motivation that led the authors in \cite{randomdayo} to study \emph{random interpolating sequences}. A random sequence in $\D^d$ is a sequence with pre-fixed deterministic radii and random arguments in $\T^d$:  let $(\theta^1_n, \dots, \theta^d_n)_{n\in\N}$ be a sequence of independent and indentically distributed random variables taking values on $\T^d$, all distributed uniformly and defined on the same probability space $(\Omega, \A, \p)$. Let $(r_n)_{n\in\N}$ be a sequence in $[0, 1)^d$, and define a random sequence $\Lambda=(\lambda_n)_{n\in\N}$ in $\D^d$ as
\[
\lambda_n(\omega)=\left(r^1_ne^{i\theta^1_n(\omega)}, \dots, r^d_ne^{i\theta^d_n(\omega)}\right), \qquad\omega\in\Omega.
\]
Looking for separation conditions on $(r_n)_{n\in\N}$ that yield almost sure interpolating properties for $\Lambda$, partition $\D^d$ into the following dyadic radial rectangles
\[
I_m:=\{z\in\D^d: 1-2^{-m_i}\leq|z^i|<1-2^{-(m_i+1)}, i=1,\dots d\}
\]
and count the number of points of $\Lambda$ that fall into each $I_m$:
\[
N_m= | \Lambda\cap I_m|,
\]
for any multi-index $m=(m_1,\dots,m_d)$ in $\N^d$. Observe that $(N_m)_{m\in\N^d}$ depends exclusively on the sequence $(r_n)_{n\in\N}$, and therefore it is a deterministic quantity. The smaller the sequence $(N_m)_{m\in\N^d}$ is, the fastest the random sequence $\Lambda$ escapes any compact subset of the polydisc, and therefore (at least intuitively) the most chances it has to be interpolating. This is the idea behind \cite[Coro 1.4]{randomdayo} ( here $|m|=m_1+\dots+m_d$ will denote the length of $m$) :
\begin{theo}
\label{theo:random}
If 
\begin{equation}
\label{eqn:ssrandom}
\sum_{m\in\N^d}N_m^{1+\frac{1}{d}}2^{-\frac{|m|}{d}}<\infty,
\end{equation}
then $\Lambda$ is interpolating almost surely. If 
\[
\sum_{m\in\N^d}N_m^22^{-|m|}=\infty,
\]
then $\Lambda$ is not interpolating almost surely.
\end{theo}
Random interpolating sequences help to construct examples of interpolating sequences of $d$-tuples of commuting matrices \emph{of any dimensions}. Let indeed $(m_n)_{n\in\N}$ be a sequence of positive integeres, and choose, for any $n$ in $\N$, $m_n$ points  $\tau_{n, 1}, \dots, \tau_{n, m_n}$ on the $d$-torus. A sequence of $d$-tuples of commuting matrices having those points as their joint spectra is $W=(W_n)_{n\in\N}$, where
\[
W_n^i:=\diag(\tau_{n, 1}^i,\dots, \tau_{n, m_n}^i).
\]
In order for the joint spectra to belong to $\D^d$, let us re-scale the matrices in $W$ via a sequence $(r_n)_{n\in\N}$ in $[0, 1)^d$:
\[
A_n:=(r_n^1W_n^1, \dots, r^d_nW_n^d)\qquad n\in\N.
\]
The more sparse the sequence $(r_n)_{n\in\N}$ is the more separated the joint spectra of the matrices in $A:=(A_n)_{n\in\N}$ are. It is natural then to ask if there is a choice of the radii $(r_n)_{n\in\N}$ that makes the sequence $A$ interpolating, for some choice of the sequence $T:=(\tau_{n, j})$. Let $(\alpha_n)_{n\in\N}$ be a sequence in $\N^d$ and let 
\[
r_n^i=1-2^{-\alpha_n^i}.
\]
If the point in $T$ are uniformly randomly chosen $\T^d$ independently one from the other and if
\begin{equation}
\label{eqn:randomatss}
\sum_{n\in\N}m_n^{1+\frac{1}{d}}~2^{-\frac{(\alpha_n^i+\dots+\alpha_n^d)}{d}}<\infty
\end{equation}
then the collection $Z$ of all joint eigenvalues of $A$ satisfies \eqref{eqn:ssrandom} and therefore Z is interpolating almost surely, thanks to Theorem \ref{theo:random}. Fixed $(m_n)_{n\in\N}$, it will suffice then to choose the sequence $(\alpha_n)_{n\in\N}$ diverging fast enough in order for \eqref{eqn:randomatss} to hold. This will give a whole family of possible choices of the parameters in $T$ for which $Z$ (and hence $A$) is interpolating. In particular, there exists a sequence of interpolating $d$-tuples of commuting matrices given any choice for their dimensions.
\section{Non Commuting $d$-tuples}
\label{sec3}
The aim of this Section is to prove Theorem \ref{theo:ncriesz} and Theorem \ref{theo:ssnc}.\\
This will be done in Section \ref{sec:ncpp} by showing that the NC Drury-Arveson space has an extension property which works as analogue of the Pick property of the classic Drury-Arveson space.\\
Section \ref{sec:ncex} will then consider an explicit example of a class of sequences of pairs of $2\times2$ non commuting matrices and it will characterize those that are interpolating.
\subsection{The NC Pick Property}
\label{sec:ncpp}
 Salomon, Shalit and Shamovich proved in \cite[Th. 4.7]{orr} that $\rkhs^2_d$ has the following non-commutative version of the Pick property: let $\Omega$ be a subset of $\Ball_d$, and let $\overline{\Omega}^{nc}$ be its NC-envelop, that is, the smallest NC set containing $\Omega$ that is also closed under left intertwiners:
\[
W\in\overline{\Omega}^{nc}\implies P^{-1}WP\in\overline{\Omega}^{nc}.
\]
 Suppose $f_0\colon\Omega\to\Mat_1$ is an  NC function that extends to an NC function on $\overline{\Omega}^{nc}$ and suppose that the map
\[
R_0\colon\spa\{K_W(u, v)\,|\, W\in\Omega\}\to\spa\{K_W(u, v)\,|\, W\in\Omega\}
\]
such that, for any $W$ in $\Omega(n)$,
\[
R_0(K_W(u, v))=K_W(u, f_0(W)^*v)\qquad u, v\in\C^n
\]
 is a contraction. Then $f_0$ extends to a contractive multiplier on $\Ball_d$, that is, there exists a function $f$ in $\rkhs^\infty_d$ whose norm doesn't exceed $1$ such that $f_{|\Omega}=f_0$.\\
With this in mind we can partially extend Theorem \ref{theo:diagonal} to this non-commutative setting, and prove Theorem \ref{theo:ncriesz}:
\begin{theo}
\label{theo:ncdiagonal}
Let $Z=(Z_n)_{n\in\N}$ be a sequence in $\Ball_d$. The following are equivalent:
\begin{description}
\item[(i)] $Z$ is interpolating;
\item[(ii)] For any bounded $(w_n)_{n\in\N}$ in $\C$ there exists a function $\phi$ in $\rkhs^\infty_d$ such that
\begin{equation}
\label{eqn:diagonal}
\phi(Z_n)=w_n~ Id\qquad n\in\N;
\end{equation}
\item[(iii)] The sequence $(\rkhs_n)_{n\in\N}$ defined in \eqref{eqn:ncmodel} is a Riesz System.
\end{description}
\end{theo}
\begin{proof}
The implication (i)$\implies$(ii) is trivial, as constant functions belong to $\rkhs^\infty_d$. Moreover, (ii)$\implies$(iii) follows from Theorem \ref{theo:nikriesz}, since for any $\phi$ in $\rkhs^\infty_d$ satisfying \eqref{eqn:diagonal} the restriction of $M^*_\phi$ to $\spa\{\rkhs_n\,|\, n\in\N\}$ is a bounded linear operator that acts like $w_n~Id$ on each $\rkhs_n$. To conclude, it suffices to show then that (iii)$\implies$(i). To do so, let $(\phi_n)_{n\in\N}$ be a bounded sequence in $\rkhs^\infty_d$. There exists an NC function $\phi_0\colon\overline{Z}^{nc}\to\Mat_1$ such that 
\[
\phi_0(Z_n)=\phi_n(Z_n)\qquad n\in\N, 
\]
since each $\phi_n$ respects direct sums and left intertwiners. Since $(\phi_n)_{n\in\N}$ is bounded, so is the sequence of operator norms $(||R_n||)_{n\in\N}$, where 
\[
R_n:=(M^*_{\phi_n})_{|\rkhs_n}
\]
is the linear map from $\rkhs_n$ to itself such that
\[
R_n(K_{Z_n}(u, v))=K_{Z_n}(u, \phi_n(Z_n)^*v)=K_{Z_n}(u, \phi_0^*(Z_n)v).
\]
Thanks to Theorem \ref{theo:nikriesz}, the map $R_0$ from $\spa\{\rkhs_n\,|\, n\in\N\}$ to itself such that
\[
R_0(K_{Z_n}(u, v))=K_{Z_n}(u, \phi_0^*(Z_n)v)
\]
is bounded, and by the NC Pick property there exists a multiplier $\phi$ in $\rkhs^\infty_d$ such that
\[
\phi(Z_n)=\phi_0(Z_n)=\phi_n(Z_n)\qquad n\in\N,
\]
concluding the proof.
\end{proof}
As for the commuting case, we would like to point out how the a-priori weaker interpolation property in (ii) is in fact equivalent for $Z$ to be interpolating.\\

Theorem \ref{theo:ncriesz} is not the only main consequence of the NC Pick property of the NC Drury-Arveson space. Indeed, since the existence of a function $\phi$ in $\rkhs^\infty_d$ of norm $M$ separating two points in $\Ball_d$
\[
\phi(Z)=Id\qquad \phi(W)=0
\]
 is equivalent to the operator $R\colon\spa\{\rkhs_W, \rkhs_Z\}\to\spa\{\rkhs_W, \rkhs_Z\}$ such that
\[
R_{|\rkhs_Z}=Id\qquad R_{|\rkhs_W}=0
\]
being bounded by $M$, one has that the Gleason NC distance between $Z$ and $W$ coincide with the sine of the angle between $\rkhs_Z$ and $\rkhs_W$ in the Hilbert space $\rkhs^2_d$:
\[
\rho_{NC}(Z, W)=\sin(\rkhs_Z, \rkhs_W).
\]
As a consequence, \eqref{eqn:ncss} can be re-written as 
\[
\prod_{n\in\N}\sin(\rkhs_n, \underset{j\ne n}{\overline{\spa}}\{\rkhs_j\})>0,
\]
 and Theorem \ref{theo:ssnc} follows from Theorem \ref{theo:ncriesz} and Proposition \ref{prop:ssriesz}.
\subsection{An Example}
\label{sec:ncex}
We give here an example of an interpolating sequence of pairs of non-commuting matrices. Let $(\alpha_n)_{n\in\N}$ and $(\beta_n)_{n\in\N}$ be two sequences of non-zero complex numbers, and define the sequence $(Z_n)_{n\in\N}$ in $\Mat_2$ as
\[
Z_n^1:=\begin{bmatrix}
0 & \alpha_n\\
0 & 0
\end{bmatrix}\quad Z_n^2:=\begin{bmatrix}
0 & 0\\
\beta_n & 0
\end{bmatrix}\qquad n\in\N.
\]
Then $Z_n^1$ and $Z_n^2$ do not commute for any $n$ in $\N$, and since
\[
Z_n^1(Z_n^1)^*+Z_n^2(Z_n^2)^*=\begin{bmatrix}
|\alpha_n|^2 & 0\\
0 & |\beta_n|^2
\end{bmatrix}
\]
we have that $Z$ belongs to $\Ball_2$ if and only if any $(\alpha_n, \beta_n)$ belongs to the bi-disc $\D^2$. We claim that a separation condition on the sequence $(\alpha_n, \beta_n)_{n\in\N}$ encodes all the cases in which $Z$ is an NC interpolating sequence:
\begin{theo}
\label{theo:ncex}
$Z$ is interpolating if and only if $(\alpha_n\beta_n)_{n\in\N}$ is interpolating in $\D$.
\end{theo}
Thanks to Theorem \ref{theo:ncdiagonal}, we need to study separation conditions of the subspaces $(\rkhs_n)_{n\in\N}$ of $\rkhs_2^2$, and according to \eqref{eqn:nckernel} this is equivalent to studying the NC kernel functions of the form $(K_{Z_n}(u_n, v_n))_{n\in\N}$, for any sequence $(u_n)_{n\in\N}$ and $(v_n)_{n\in\N}$ in $\C^2$. Observe that 
\[
(Z_n^i)^2=0\qquad n\in\N, i=1, 2
\]
and therefore, out of all the coefficients of the NC function
\begin{equation}
\label{eqn:ncpsex}
K_{Z_n}(u_n, v_n)(Z)=\sum_{l\in\W_2}\braket{v_n, Z_n^lu_n}Z^l\qquad Z\in\Ball_2,
\end{equation}
the only non zero ones are the ones associated to a word which alternates its two letters $a$ and $b$. Such words can be written in four different ways:
\[
\begin{split}
\omega&=(ab)^l\qquad l\in\N\\
\omega&=(ba)^l\qquad l\in\N\\
\omega&=a(ba)^l\qquad l\in\N\\
\omega&=b(ab)^l\qquad l\in\N,
\end{split}
\]
depending on its lenght and its first and last letter. Since
\[
\begin{split}
Z_n^1Z_n^2&=\begin{bmatrix}
\alpha_n\beta_n & 0\\
0 & 0
\end{bmatrix}\quad Z_n^2Z_n^1=\begin{bmatrix}
0 & 0\\
0 & \alpha_n\beta_n
\end{bmatrix}\\
Z_n^2Z_n^1Z_n^2&=\begin{bmatrix}
0 & 0\\
\beta_n\alpha_n\beta_n & 0
\end{bmatrix}\quad Z_n^1Z_n^2Z_n^1=\begin{bmatrix}
0 & 0\\
0& \alpha_n\beta_n\alpha_n
\end{bmatrix},
\end{split}
\]
we can split the NC power series in \eqref{eqn:ncpsex} into four {\bf orthogonal} pieces
\[
\begin{split}
K_{Z_n}(u_n, v_n)(Z)=&\braket{v_n, u_n}+v_n^1\overline{u_n^1}~\sum_{l=1}^\infty(\overline{\alpha_n\beta_n})^l(Z^1Z^2)^l+v_n^2\overline{u_n^2}~\sum_{l=1}^\infty(\overline{\alpha_n\beta_n})^l(Z^2Z^1)^l+\\
+&\overline{\alpha_n}v_n^1\overline{u_n^2}~\sum_{l=1}^\infty(\overline{\alpha_n\beta_n})^lZ^1(Z^2Z^1)^l+\overline{\beta_n}v_n^2\overline{u_n^1}~\sum_{l=1}^\infty(\overline{\alpha_n\beta_n})^lZ^2(Z^1Z^2)^l\\
=&v_n^1\overline{u_n^1}~\sum_{l=0}^\infty(\overline{\alpha_n\beta_n})^l(Z^1Z^2)^l+v_n^2\overline{u_n^2}~\sum_{l=0}^\infty(\overline{\alpha_n\beta_n})^l(Z^2Z^1)^l+\\
+&\overline{\alpha_n}v_n^1\overline{u_n^2}~\sum_{l=1}^\infty(\overline{\alpha_n\beta_n})^lZ^1(Z^2Z^1)^l+\overline{\beta_n}v_n^2\overline{u_n^1}~\sum_{l=1}^\infty(\overline{\alpha_n\beta_n})^lZ^2(Z^1Z^2)^l.\\
=&\sum_{m, j=1}^2v^m_n\overline{u_n^j}K_{Z_n}(e^m, e^j)(Z),
\end{split}
\]
where $\{e^1, e^2\}$ is the standard basis of $\C^2$.
\begin{proof}[Proof of Theorem \ref{theo:ncex}]
The sequence $(\hat{K}_n(e^1, e^1))_{n\in\N}$ of normalized NC kernels can be unitarily identified with a multiple of the sequence $(s_{\alpha_n\beta_n})_{n\in\N}$ of Szegö kernels at the points of $(\alpha_n\beta_n)_{n\in\N}$ in $\D$. Thus if $(\alpha_n\beta_n)_{n\in\N}$ is not interpolating, the sequence $(\hat{s}_{\alpha_n\beta_n})_{n\in\N}$ is not a Reisz system, and therefore $(\hat{K}_n(e^1, e^1))_{n\in\N}$ is not a Riesz system either. Hence, thanks to Theorem \ref{theo:ncdiagonal}, $Z$ is not interpolating.\\
Conversely, assume that $(\alpha_n\beta_n)_{n\in\N}$ is an interpolating sequences in $\D$ or, equivalently, that the associated sequence $(\hat{s}_{\alpha_n\beta_n})_{n\in\N}$ of Szegö kernels in $\h^2$ is a Riesz system. Both the sequences $(\hat{K}_{Z_n}(e^1, e^1))_{n\in\N}$ and $(\hat{K}_{Z_n}(e^2, e^2))$ are untarly equivalent to  $(\hat{s}_{\alpha_n\beta_n})_{n\in\N}$ in $\h^2$, and hence are Riesz systems. The same holds for $(\hat{K}_{Z_n}(e^1, e^2))_{n\in\N}$ and $(\hat{K}_{z_n}(e^2, e^1))_{n\in\N}$, since they are both unitarly equivalent to a multiple of the sequence of shifted Szegö kernels $(zs_{\alpha_n\beta_n})_{n\in\N}$. Since the four sequences \[
(K_{Z_n}(u^m, u^j))_{n\in\N}\qquad m,j=1, 2
\]
are pairwise orthogonal in $\rkhs^2_2$, we can conclude that the sequence of subspaces
\[
\rkhs_n=\{K_{Z_n}(u_n, v_n)\,|\, u_n, v_n\in\C^2\}\qquad n\in\N
\]
is a Riesz system in $\rkhs^2_2$, thanks to Lemma \ref{lemma:orthriesz} below.
\end{proof}
\begin{lemma}
\label{lemma:orthriesz}
Let $m$ be a finite positive integer, and let $X^i:=(x^i_n)_{n\in\N}$, $i=1, \dots, m$ be $m$ pairwise orthogonal sequences in a Hilbert space $\rkhs$. If each $X^i$ is a Riesz system, then so is the sequence $X=(x_n)_{n\in\N}$ defined as
\[
x_n:=x^1_n+\dots+x^m_n\qquad n\in\N.
\]
\end{lemma}
\begin{proof}
Without loss of generality, assume that $||x_n||=1$ for any $n$ in $\N$. Thus
\[
x_n=t^1_n\hat{x}^1_n+\dots+t^m_n\hat{x}^m_n,
\]
where thanks to orthogonality 
\begin{equation}
\label{eqn:exlemmasum}
\sum_{i=1}^m|t^i_n|^2=1\qquad n\in\N.
\end{equation}
Fix then an arbitrary $(a_n)_{n\in\N}$ in $l^2$, and observe that
\[
\left|\left|\sum_{n\in\N}a_nx_n\right|\right|^2=\sum_{i=1}^m\left|\left|\sum_{n\in\N}a_nt^i_n\hat{x}^i_n\right|\right|^2.
\]
Therefore, if $C_i$ is the Riesz bound of each $X^i$,
\[
\left|\left|\sum_{n\in\N}a_nx_n\right|\right|^2\leq\sum_{i=1}^mC_i^2~\sum_{n\in\N}|t^i_n|^2|a_n|^2\leq\max_{i=1,\dots, m}C_i^2\sum_{n\in\N}|a_n|^2
\]
and 
\[
\frac{1}{\underset{i=1, \dots, m}{\max} C_i^2}~\sum_{n\in\N}|a_n|^2\leq\sum_{i=1}^m\frac{1}{C_i^2}\sum_{n\in\N}|t_n^i|^2|a_n|^2\leq\left|\left|\sum_{n\in\N}a_nx_n\right|\right|^2,
\]
thanks to \eqref{eqn:exlemmasum}.
\end{proof}

\end{document}